\documentclass[reqno,12pt]{amsart}

\usepackage{amsmath, epsfig}
\usepackage{amssymb}
\usepackage{amsfonts}
\usepackage{latexsym}
\usepackage{amsthm}
\usepackage{cite}
\newtheorem{theorem}{Theorem}

\newtheorem{lemma}{Lemma}

\theoremstyle{remark}

\numberwithin{equation}{section}

\newcommand{\Z}{\mathbb Z}

\allowdisplaybreaks

\author{Xiaoxia Wang}
\address{DEPARTMENT OF MATHEMATICS, SHANGHAI UNIVERSITY, SHANGHAI 200444, P. R. CHINA}
\email{$^*$ Corresponding author. xiaoxiawang@shu.edu.cn (X. Wang), xchangi@shu.edu.cn (C. Xu).}

\author{Chang Xu$^*$}
\title{New $q$-supercongruences from the Bailey transformation }


\begin{document}
\begin{abstract}
Inspired by the recent work of Guo, we establish some new $q$-supercongruences including $q$-analogues of some Ramannujan-type supercongruences, by using the Bailey transformation formula and the `creative microscoping' method recently introduced by Guo and Zudilin.
\end{abstract}
\maketitle
\section{introduction}

It was van Hamme \cite{Van} that investigated partial sums of Ramanujan's summation formular for $1/\pi$
and conjectured $13$ Ramanujan-type supercongruences labeled as (A.2)-(M.2), such as:
\begin{equation}\label{eq5}
\begin{split}
&\mathrm{(B.2)}\quad\sum_{k=0}^{(p-1)/2}(-1)^k(4k+1)\frac{(\frac{1}{2})^3_k}{k!^3}\equiv\frac{-p}{\Gamma_p(\frac{1}{2})^2}\pmod{p^3};\\
&\mathrm{(E.2)}\quad\sum_{k=0}^{(p-1)/3}(-1)^k(6k+1)\frac{(\frac{1}{3})^3_k}{k!^3}\equiv p\pmod{p^3}
\quad\mathrm{if}\quad p\equiv 1\pmod6;\\
&\mathrm{(F.2)}\quad\sum_{k=0}^{(p-1)/4}(-1)^k(8k+1)\frac{(\frac{1}{4})^3_k}{k!^3}\equiv\frac{-p}{\Gamma_p(\frac{1}{4})\Gamma_p(\frac{3}{4})}\pmod{p^3}.
\end{split}
\end{equation}
Here and in what follows, $p$ is an odd prime, $(a)_n=a(a+1)\cdots(a+n-1)$ is  the Pochhammer symbol, and
 the $p$-adic gamma function is defined as
 $$\Gamma_p(x)=\lim_{m\rightarrow x}(-1)^m\prod_{\substack{0< k< m\\ (k,p)=1}}k,$$
 where the limit is for $m$ tending to $x$ $p$-adically in $\Z_{\ge 0}$.
In 2015, Swisher \cite{Swisher} proved the following unified form of
van Hamme's three supercongruences in \eqref{eq5}:
\begin{equation}\label{eq6}
\sum_{k=0}^{a(b p-1)}(-1)^{k}\left(\frac{2 k}{a}+1\right) \frac{(a)_{k}^{3}}{k !^{3}}
\equiv(-1)^{a(b p-1)} p \cdot b=\frac{-p b}{\Gamma_{p}(a) \Gamma_{p}(1-a)} \pmod {p^3},
\end{equation}
where
 $a \in\left\{\frac{1}{2}, \frac{1}{3}, \frac{1}{4}\right\}$, $p \geq 5$ when $a=1/4$,
 $b=1$ when $p\equiv 1 \pmod {\frac{1}{a}}$ and $b=\frac{1}{a}-1$ when $ p \equiv-1\pmod {\frac{1}{a}}$.

For completeness, we first present the needed definitions. For any complex numbers $a$ and $q$, the
{\em $q$-shifted factorial} is defined as
\begin{align*}
    (a;q)_\infty=\prod_{j=0}^{\infty}(1-aq^j) \quad \mathrm{and} \quad (a;q)_k = \frac{(a;q)_\infty}{(aq^k;q)_\infty}\quad \mathrm{where}\quad k\in \Z.
    \end{align*}
    For convenience, we adopt the notation
    \begin{align*}
    (a_1,a_2 ,...,a_r;q)_k=(a_1;q)_k(a_2;q)_k \cdots(a_r;q)_k,  \quad k \in \mathbb {C} \cup \infty.
    \end{align*}
Following Gasper and Rahman \cite{Gasper}, the basic hypergeometric series $_{r+1}\phi_r$ is defined by
\begin{equation}\label{}
\begin{split}
_{r+1}\phi_{r}\left[\begin{array}{c}
a_1,a_2,\ldots,a_{r+1}\\
b_1,b_2,\ldots,b_{r}
\end{array};q,\, z
\right]
=\sum_{k=0}^{\infty}\frac{(a_1,a_2,\ldots, a_{r+1};q)_k}
{(q,b_1,\ldots,b_{r};q)_k} z^k.
\end{split}
\end{equation}
Let $[n]_q=[n]=(1-q^n)/(1-q)$ be the $q$-{\em integer} and $\Phi_n(q)$ the $n$-th {\em cyclotomic polynomial} in $q$:
\begin{align*}
\Phi_n(q)=\prod_{\substack{1\leqslant k\leqslant n\\ \gcd(n,k)=1}}(q-\zeta^k),
\end{align*}
where $\zeta$ is an $n$-th primitive root of unity.

In recent years, a number of authors have studied Ramanujan-type supercongruences from the perspective of $q$-analogues.
There are various methods to prove $q$-congruences, such as, properties of roots of unity \cite{Ni},
basic hypergeometric series identities \cite{WeiChuanan,Wei2,Long,Liu1,Liu_Wang1,Wang_Yu}, the $q$-Wilf-Zeilberger method \cite{Tauraso,Guo1}.
For instance, using the Bailey transformation formula (see \cite[Proposition 11.4]{Guo}): for $|\alpha^2ab/q^2|<1$,
\begin{flalign}\label{Eq2}
&\sum_{k=0}^{\infty}\frac{(1+\alpha q^{2k})(\alpha^2,q^2/a,q^2/b;q^2)_k(-q,\alpha q/\lambda;q)_k(-\alpha^2\lambda ab)^kq^{k^2-3k}}
{(1+\alpha)(q^2,\alpha^2a,\alpha^2b;q^2)_k(\alpha,-\lambda;q)_k}\nonumber\\
&=\frac{(\alpha^2q^2,\alpha^2ab/{q^2};q^2)_\infty}{(\alpha^2a,\alpha^2b;q^2)_\infty}
\,{}_{4}\phi_{3}\!\left[\begin{array}{c}
q^2/a,\ q^2/b,\ \lambda,\ q\lambda \\
q\alpha,\, q^2\alpha,\, \lambda^2
\end{array};q^2,\, \frac{\alpha^2ab}{q^2}
\right],
\end{flalign}
 Guo \cite{Guo} proved some $q$-supercongruences as follows: for odd integers $n$,
\begin{flalign}
&\sum_{k=0}^{(n-1) / 2}(-1)^{k} \frac{\left(1+q^{4 k+1}\right)\left(q^{2} ; q^{4}\right)_{k}^{3}}{(1+q)\left(q^{4}
 ; q^{4}\right)_{k}^{3}} q^{2 k^{2}+2 k}\nonumber\\
&\quad\equiv(-1)^{(n-1) / 2} q^{(n-1)^{2} / 2}[n]_{q^{2}}\sum_{k=0}^{(n-1)/2} \frac{\left(q^{2} ; q^{4}\right)_{k}^{3}
 q^{2 k}}{\left(q^{4}, q^{3},q^{5}; q^{4}\right)_{k}} \pmod{\Phi_{n}(-q)^{3}\Phi_{n}(q)^{2}};\label{eq8}\\
&\sum_{k=0}^{(n+1) / 2}(-1)^{k} \frac{\left(1+q^{4k-1}\right)\left(q^{-2};q^{4}\right)_{k}^{3}}
{\left(1+q^{-1}\right)\left(q^{4};q^{4}\right)_{k}^3} q^{2k^2+6k}\nonumber\\
&\quad\equiv(-1)^{(n-1)/2}q^{(n-1)^2/2}[n]_{q^2}\sum_{k=0}^{(n+1)/2} \frac{\left(q^{-2}; q^{4}\right)_{k}^2
\left(q^2 ; q^{4}\right)_{k}}{\left(q^{4} ,q , q^{3} ; q^{4}\right)_{k}} q^{6k}\pmod{\Phi_{n}(-q)^{3}\Phi_{n}(q)^{2}}.\label{eq13}
\end{flalign}
Clearly, van Hamme's supercongruence (B.2) follows from the $q$-supercongruence \eqref{eq8} with $n=p$ and $q\rightarrow-1$.
Also, when $n=p$ and $q\rightarrow1$, \eqref{eq8} reduces to
the following supercongruence:
\begin{equation*}
\sum_{k=0}^{(p-1)/2}(-1)^{k} \frac{(\frac{1}{2})_{k}^{3}}{k !^{3}}
\equiv(-1)^{(p-1)/2} p \sum_{k=0}^{(p-1)/2}\frac{(\frac{1}{2})_{k}^{3}}{k!(\frac{3}{4})_{k}(\frac{5}{4})_{k}} \pmod {p^2},
\end{equation*}
which is an equivalent form of  the supercongruence proved by He \cite{He}.

Inspired by the above work, we establish some new Ramanujan-type $q$-supercongruences in the following results.
 \begin{theorem}\label{Thm3}
Let $n\equiv s\pmod3$ be a positive odd integer with $s\in\left\{1,-1\right\}$. Then,
modulo $\Phi_n(-q)^3\Phi_n(q)^2$,
\begin{align}\label{EE1}
&\sum_{k=0}^{(n-s)/3}(-1)^{k}
\frac{\left(1+q^{6k+s}\right)\left(q^{2s};q^{6}\right)_{k}^{3}}{\left(1+q^s\right)\left(q^{6} ; q^{6}\right)_{k}^{3}} q^{3k^2+2k(3-s)}\nonumber\\
&\quad\equiv(-1)^{({n-s})/{3}}\frac{[n]_{q^2}}{[s]_{q^2}}q^{{(n-s)(n+s-3)}/{3}}\sum_{k=0}^{(n-s)/3} \frac{\left(q^{2s} ; q^{6}\right)_{k}^{2}\left(q^{3} ; q^{6}\right)_{k}}{\left(q^{6},q^{3+s} ,q^{6+s} ; q^{6}\right)_{k}} q^{2k(3-s)}.
\end{align}
\end{theorem}
Performing $n=p$, $s=-1$, $q\to 1$; $n=p$, $s=1$, $q\to 1$; $n=p$, $s=-1$, $q\to -1$ and $n=p$, $s=1$, $q\to -1$ in Theorem \ref{Thm3} respectively, we get some new supercongruences as follows:
\begin{align*}
&\sum_{k=0}^{(p+1)/3}(-1)^k\frac{(-\frac{1}{3})^3_k}{k!^3}\equiv-(-1)^{(p+1)/3}p
\sum_{k=0}^{(p+1)/3}\frac{(-\frac{1}{3})^2_k(\frac{1}{2})_k}{k!(\frac{1}{3})_k(\frac{5}{6})_k}\pmod{p^2},\\
&\sum_{k=0}^{(p-1)/3}(-1)^k\frac{(\frac{1}{3})^3_k}{k!^3}\equiv (-1)^{(p-1)/3}p
\sum_{k=0}^{(p-1)/3}\frac{(\frac{1}{3})^2_k(\frac{1}{2})_k}{k!(\frac{2}{3})_k(\frac{7}{6})_k}\pmod{p^2},\\
&\sum_{k=0}^{(p+1)/3}(6k-1)\frac{(-\frac{1}{3})^3_k}{k!^3}\equiv p
\sum_{k=0}^{(p+1)/3}\frac{(-\frac{1}{3})^2_k}{k!(\frac{1}{3})_k}\pmod{p^3},\\
&\sum_{k=0}^{(p-1)/3}(6k+1)\frac{(\frac{1}{3})^3_k}{k!^3}\equiv p
\sum_{k=0}^{(p-1)/3}\frac{(\frac{1}{3})^2_k}{k!(\frac{2}{3})_k}\pmod{p^3},
\end{align*}
where the last congruence is similar to van Hamme's (E.2).
\begin{theorem}\label{Thm4}

Let $n\equiv s\pmod4$ be a positive integer with $s\in\left\{1,-1\right\}$. Then,
modulo $\Phi_n(-q)^3\Phi_n(q)^2$,
\begin{equation}\label{EE2}
\begin{split}
&\sum_{k=0}^{(n-s)/4}(-1)^{k}
\frac{\left(1+q^{8k+s}\right)\left(q^{2s};q^{8}\right)_{k}^{3}}{\left(1+q^s\right)\left(q^{8} ; q^{8}\right)_{k}^{3}} q^{4k^2+2k(4-s)}\\
&\quad\equiv(-1)^{({n-s})/{4}}\frac{[n]_{q^2}}{[s]_{q^2}}q^{{(n-s)(n+s-4)}/{4}}\sum_{k=0}^{(n-s)/4} \frac{\left(q^{2s} ; q^{8}\right)_{k}^{2}\left(q^{4} ; q^{8}\right)_{k}}{\left(q^{8},q^{4+s} ,q^{8+s} ; q^{8}\right)_{k}} q^{2k(4-s)}.
\end{split}
\end{equation}
\end{theorem}
Choosing $n=p$, $s=-1$, $q\to 1$; $n=p$, $s=1$, $q\to 1$; $n=p$, $s=-1$, $q\to -1$ and $n=p$, $s=1$, $q\to -1$ in Theorem \ref{Thm4} respectively, we deduce the following supercongruences:
\begin{align*}
&\sum_{k=0}^{(p+1)/4}(-1)^k\frac{(-\frac{1}{4})^3_k}{k!^3}\equiv-(-1)^{(p+1)/4}p
\sum_{k=0}^{(p+1)/4}\frac{(-\frac{1}{4})^2_k(\frac{1}{2})_k}{k!(\frac{3}{8})_k(\frac{7}{8})_k}\pmod{p^2},\\
&\sum_{k=0}^{(p-1)/4}(-1)^k\frac{(\frac{1}{4})^3_k}{k!^3}\equiv (-1)^{(p-1)/4}p
\sum_{k=0}^{(p-1)/4}\frac{(\frac{1}{4})^2_k(\frac{1}{2})_k}{k!(\frac{5}{8})_k(\frac{9}{8})_k}\pmod{p^2},\\
&\sum_{k=0}^{(p+1)/4}(-1)^k(8k-1)\frac{(-\frac{1}{4})^3_k}{k!^3}\equiv (-1)^{(p+1)/4}p\pmod{p^3},\\
&\sum_{k=0}^{(p-1)/4}(-1)^k(8k+1)\frac{(\frac{1}{4})^3_k}{k!^3}\equiv (-1)^{(p-1)/4}p\pmod{p^3},
\end{align*}
where the last congruence is equivalent to van Hamme's (F.2) by using the following property of the $p$-adic gamma function: for odd primes $p$ and $x\in \mathbb Z_p$,
\begin{equation*}
\Gamma_p(x)\Gamma_p(1-x)=(-1)^{a_p(x)},
\end{equation*}
where $a_p(x)\in\left\{1,2,\cdots,p\right\}$ with $x\equiv a_p(x)\pmod p$.

In fact, we can combine Guo's $q$-supercongruences \eqref{eq8}, \eqref{eq13} and
Theorems \ref{Thm3}-\ref{Thm4} in the unified form as follows.
\begin{theorem}\label{Thm1}
Let $n$  be an odd integer and $d$  a positive integer with $\gcd(n,d)=1$. Let $r$ be an integer with $n-dn+d\leq r\leq n$  and $n\equiv r\pmod d$. Then,
modulo $\Phi_n(-q)^3\Phi_n(q)^2$,
\begin{equation}\label{Eq1}
\begin{split}
&\sum_{k=0}^{(n-r)/d}(-1)^k\frac{(1+q^{2dk+r})(q^{2r};q^{2d})^3_k}{(1+q^r)(q^{2d};q^{2d})^3_k}q^{dk^2+2k(d-r)}\\
&\quad\equiv(-1)^{(n-r)/d}\frac{[n]_{q^2}}{[r]_{q^2}}q^{{(n-r)(n+r-d)}/d}
\sum_{k=0}^{(n-r)/d}\frac{(q^{2r};q^{2d})^2_k(q^d;q^{2d})_k}{(q^{2d},q^{d+r},q^{2d+r};q^{2d})_k}q^{2k(d-r)}.
\end{split}
\end{equation}
\end{theorem}
Letting $d=2$, $r=1$; $d=2$, $r=-1$; $d=3$, $r=s$ and $d=4$, $r=s$ in Theorem \ref{Thm1} respectively, we immediately get
the $q$-supercongruences \eqref{eq8}, \eqref{eq13}, \eqref{EE1} and \eqref{EE2}.

Moreover, when $n=p$, $r=1$, $d$ is a positive even integer and $q\to -1$ in \eqref{Eq1}, we obtain the following supercongruence:
\begin{equation}\label{Eq7}
\sum_{k=0}^{(p-1)/d}(-1)^{k}\left(2dk+1\right)\frac{(1/d)_{k}^{3}}{k!^{3}}
\equiv(-1)^{(p-1)/d}p \pmod {p^3},
\end{equation}
 which is just an extension of \eqref{eq6} for $p\equiv1\pmod{\frac{1}{a}}$ and the special case of Guo's supercongruence \cite[Eq. (1.7)]{Guo1} with $r=1$.

The rest of the paper is organized as follows. In Section \ref{Sec3}, we will prove Theorem \ref{Thm1}
by using the Bailey transformation formula \eqref{Eq2} together with the `creative microscoping' method which was
introduced by Guo and Zudilin \cite{GuoZu}.
Then, we shall give another $q$-supercongruence from the Bailey transformation formula in Section \ref{Sec4}.

\section{Proof of Theorem \ref{Thm1}}\label{Sec3}
We first present an auxiliary $q$-congruence proved by Guo and Schlosser \cite[Lemma 2.1]{Guo2}:
\begin{lemma}\label{Lemma1}
Let $n$, $d$ be positive integers, $r$ an integer satisfying $n-dn+d\leq r\leq n$  and $n\equiv r\pmod d$. Then, for $0\le k\le (n-r)/d$,
we have
\begin{equation}\label{eq1}
\frac{(aq^r;q^d)_{(n-r)/d-k}}{(q^d/a;q^d)_{(n-r)/d-k}}
\equiv (-a)^{(n-r)/d-2k}\frac{(aq^r;q^d)_k}{(q^d/a;q^d)_k}q^{{(n-r)(n-d+r)}/{(2d)}+k(d-r)}\pmod{\Phi_n(q)}.
\end{equation}
\end{lemma}
We next present the following one-parametric generalization of Theorem \ref{Thm1}.
\begin{lemma}\label{Lemma2}
Let $n$ be an odd integer, $d$ a positive integer, $a$ an indeterminate, $r$ an integer with $n-dn+d\leq r\leq n$
 and $n\equiv r\pmod d$.
Then, modulo $\Phi_n(-q)(1-aq^{2n})(a-q^{2n})$,
\begin{flalign}
&\sum_{k=0}^{(n-r)/d}(-1)^k\frac{(1+q^{2dk+r})(aq^{2r},q^{2r}/a,q^{2r};q^{2d})_k}
{(1+q^r)(aq^{2d},q^{2d}/a,q^{2d};q^{2d})_k}q^{dk^2+2k(d-r)}\nonumber\\
&\quad\equiv(-1)^{(n-r)/d}\frac{[n]_{q^2}}{[r]_{q^2}}q^{{(n-r)(n+r-d)}/d}
\sum_{k=0}^{(n-r)/d}\frac{(aq^{2r},q^{2r}/a,q^d;q^{2d})_k}
{(q^{2d},q^{d+r},q^{2d+r};q^{2d})_k}q^{2k(d-r)}.\label{eq2}
\end{flalign}
\end{lemma}
\begin{proof}
For $a=q^{2n}$ or $a=q^{-2n}$, the left-hand side of \eqref{eq2} is equal to
\begin{flalign*}
&\sum_{k=0}^{(n-r)/d}(-1)^k\frac{(1+q^{2dk+r})(q^{2r+2n},q^{2r-2n},q^{2r};q^{2d})_k}
{(1+q^r)(q^{2d+2n},q^{2d-2n},q^{2d};q^{2d})_k}q^{dk^2+2k(d-r)}\\
&\quad=\frac{(q^{2r+2d},q^{2d-2r};q^{2d})_\infty}{(q^{2d-2n},q^{2d+2n};q^{2d})_\infty}
\sum_{k=0}^{(n-r)/d}\frac{(q^{2r+2n},q^{2r-2n},q^d;q^{2d})_k}
{(q^{2d},q^{d+r},q^{2d+r};q^{2d})_k}q^{2k(d-r)}
\end{flalign*}
which is computed by the Bailey transformation \eqref{Eq2} with the parameter substitutions $q\rightarrow q^d$, $\alpha=q^r$,
 $a=q^{2(d-r-n)}$, $b=q^{2(d-r+n)}$ and $\lambda=q^d$. It is clear that
 \begin{equation}\label{eq3}
 \frac{(q^{2r+2d},q^{2d-2r};q^{2d})_\infty}{(q^{2d-2n},q^{2d+2n};q^{2d})_\infty}
 =\frac{(q^{2r+2d};q^{2d})_{(n-r)/d}}{(q^{2d-2n};q^{2d})_{(n-r)/d}}
 =(-1)^{(n-r)/d}\frac{[n]_{q^2}}{[r]_{q^2}}q^{(n-r)(n+r-d)/d}.
 \end{equation}
 This proves that the congruence \eqref{eq2} holds modulo $1-aq^{2n}$ or $a-q^{2n}$.

 From  $q$-congruence \eqref{eq1}, we deduce that, for $0\le k\le (n-r)/d$,
 the $k$-th and $((n-r)/d-k)$-th terms of the left-hand side of \eqref{eq2} modulo $\Phi_n(-q)$ cancel each other,
 i.e.,
 \begin{flalign*}
 &(-1)^{(n-r)/d-k}\frac{(1+q^{2n-2dk-r)})(aq^{2r},q^{2r}/a,q^{2r};q^{2d})_{(n-r)/d-k}}
 {(1+q^r)(q^{2d}/a,aq^{2d},q^{2d};q^{2d})_{(n-r)/d-k}}q^{d((n-r)/d-k)^2+2((n-r)/d-k)(d-r)}\\
 &\quad\equiv-(-1)^k\frac{(1+q^{2dk+r})(aq^{2r},q^{2r}/a,q^{2r};q^{2d})_k}
{(1+q^r)(aq^{2d},q^{2d}/a,q^{2d};q^{2d})_k}q^{dk^2+2k(d-r)}\pmod{\Phi_n(-q)}.
 \end{flalign*}
 Thus, the left-hand side of \eqref{eq2} is congruent to $0$ modulo $\Phi_n(-q)$.
  It is obvious that the right-hand side of \eqref{eq2} is also congruent to $0$ modulo $\Phi_n(-q)$.
   As a result, the congruence \eqref{eq2} holds  modulo $\Phi_n(-q)$.
   Since $1-aq^{2n}$, $a-q^{2n}$ and $\Phi_n(-q)$ are pairwise relatively prime polynomials, the proof of \eqref{eq2} is completed.
\end{proof}

\begin{proof}[Proof of Theorem \ref{Thm1}]
The denominators of the left-hand side of \eqref{eq2} are relatively prime to $\Phi_n(q^2)$ as $a\rightarrow 1$ because of $\gcd(n,d)=1$
and $n\equiv r \pmod{d}$.
The denominators of the reduced form of the right-hand side of \eqref{eq2} are relatively prime to $\Phi_n(q^2)$
 because of  the factor $[n]_{q^2}$ before the summation. On the other hand, the limit of $(1-aq^{2n})(a-q^{2n})$
 as $a\rightarrow 1$ has the factor $\Phi_n(q^2)^2$. Thus, letting $a\rightarrow1$ in \eqref{eq2}, we obtain the $q$-congruence \eqref{Eq1}.
\end{proof}

\section{The other $q$-supercongruence \label{Sec4}}

In \cite{Guo}, Guo also proved the following two $q$-supercongruences: for odd integers $n$,

\begin{equation}\label{E3}
\begin{split}
&\sum_{k=0}^{(n-1)/2}(-1)^{k} \frac{\left(1+q^{4 k+1}\right)\left(q^{2} ; q^{4}\right)_{k}^{2}}{(1+q)\left(q^{4} ; q^{4}\right)_{k}^{2}} q^{2k^{2}+k} \equiv(-1)^{(n-1) / 2}[n]_{q^{2}}q^{(n-1)^{2}/2}\\
&\:\:\times\sum_{k=0}^{(n-1)/2} \frac{\left(q^{2} ; q^{4}\right)_{k}q^{2k}}{[4k+1]\left(q^{4};q^{4}\right)_{k}} \left\{
\begin{aligned}
\pmod {\Phi_{n}(q)^{2} \Phi_{n}(-q)^{3}} \qquad \mathrm{if} \quad n\equiv1\pmod4,\\
\pmod{\Phi_{n}(q)^{3} \Phi_{n}(-q)^{3}} \qquad\mathrm{if}\quad n\equiv3\pmod4;
\end{aligned}
\right.
\end{split}
\end{equation}
\begin{equation}\label{E4}
\begin{split}
&\sum_{k=0}^{(n+1) / 2}(-1)^{k} \frac{\left(1+q^{4k-1}\right)\left(q^{-2};q^{4}\right)_{k}^{2}}
{\left(1+q^{-1}\right)\left(q^{4};q^{4}\right)_{k}^2} q^{2k^2+3k}
\equiv(-1)^{(n-1) / 2}[n]_{q^2}q^{(n-1)^{2}/2}\\
&\:\:\times \sum_{k=0}^{(n+1)/2} \frac{\left(q^{-2},q^{-1}; q^{4}\right)_{k}}{(q^{4},q^3 ; q^{4})_{k}} q^{6k}\left\{
\begin{aligned}
\pmod {\Phi_{n}(q)^{2} \Phi_{n}(-q)^{3}}, \qquad \mathrm{if} \quad n\equiv1\pmod4,\\
\pmod{\Phi_{n}(q)^{3} \Phi_{n}(-q)^{3}}, \qquad \mathrm{if} \quad n\equiv3\pmod4.
\end{aligned}
\right.
\end{split}
\end{equation}
In this section, we will give a partial generalization of \eqref{E3} and \eqref{E4}.
\begin{theorem}\label{Thm2}
Let $n$ be an odd integer, $d$ a positive integer with $\gcd(n,d)=1$, $r$ an integer with $n-dn+d\leq r\leq n$  and $n\equiv r\pmod d$.
Then, modulo $\Phi_n(-q)^3\Phi_n(q)^2$,
\begin{equation}\label{EQ2}
\begin{split}
&\sum_{k=0}^{(n-r)/d}(-1)^k\frac{(1+q^{2dk+r})(q^{2r};q^{2d})^2_k}
{(1+q^r)(q^{2d};q^{2d})^2_k}q^{dk^2+k(d-r)}\\
&\quad\equiv(-1)^{(n-r)/d}\frac{[n]_{q^2}}{[r]_{q^2}}q^{{(n-r)(n+r-d)}/d}
\sum_{k=0}^{(n-r)/d}\frac{(q^{2r},q^r;q^{2d})_k}
{(q^{2d},q^{2d+r};q^{2d})_k}q^{2k(d-r)}.
\end{split}
\end{equation}
\end{theorem}
\begin{proof}[Sketch of proof]
Applying the Bailey transformation formula \eqref{Eq2} by performing the replacements $q\rightarrow q^d$, $\alpha=q^r$,
 $a=q^{2(d-r-n)}$, $b=q^{2(d-r+n)}$ and $\lambda=q^r$, we have
\begin{equation*}
\begin{split}
&\sum_{k=0}^{(n-r)/d}(-1)^k\frac{(1+q^{2dk+r})(q^{2r+2n},q^{2r-2n};q^{2d})_k}
{(1+q^r)(q^{2d+2n},q^{2d-2n};q^{2d})_k}q^{dk^2+k(d-r)}\\
&\quad=\frac{(q^{2r+2d},q^{2d-2r};q^{2d})_\infty}{(q^{2d-2n},q^{2d+2n};q^{2d})_\infty}
\sum_{k=0}^{(n-r)/d}\frac{(q^{2r+2n},q^{2r-2n},q^r;q^{2d})_k}
{(q^{2d},q^{2d+r},q^{2r};q^{2d})_k}q^{2k(d-r)}.
\end{split}
\end{equation*}
Using equation \eqref{eq3}, we immediately get, modulo $1-aq^{2n}$ or $a-q^{2n}$,
\begin{equation}\label{EQ1}
\begin{split}
&\sum_{k=0}^{(n-r)/d}(-1)^k\frac{(1+q^{2dk+r})(aq^{2r},q^{2r}/a;q^{2d})_k}
{(1+q^r)(aq^{2d},q^{2d}/a;q^{2d})_k}q^{dk^2+k(d-r)}\\
&\quad\equiv(-1)^{(n-r)/d}\frac{[n]_{q^2}}{[r]_{q^2}}q^{{(n-r)(n+r-d)}/d}
\sum_{k=0}^{(n-r)/d}\frac{(aq^{2r},q^{2r}/a,q^r;q^{2d})_k}
{(q^{2d},q^{2dk+r},q^{2r};q^{2d})_k}q^{2k(d-r)}.
\end{split}
\end{equation}
Since $1-aq^{2n}$ and $a-q^{2n}$ are relatively
prime polynomials, \eqref{EQ1} still holds modulo $(1-aq^{2n})(a-q^{2n})$.

From  $q$-congruence (\ref{eq1}), we deduce that, for $0\le k\le (n-r)/d$,
the $((n-r)/d-k)$-th terms of the left-hand side of (\ref{EQ1}) modulo $\Phi_n(-q)$ can be written as follows,
 \begin{equation*}
 \begin{split}
 &(-1)^{(n-r)/d-k}\frac{(1+q^{2n-2dk-r)})(aq^{2r},q^{2r}/a;q^{2d})_{(n-r)/d-k}}
 {(1+q^r)(q^{2d}/a,aq^{2d};q^{2d})_{(n-r)/d-k}}q^{d((n-r)/d-k)^2+((n-r)/d-k)(d-r)}\\
 &\equiv(-1)^{(n-r)/d-k}\frac{(1+q^{2dk+r})(aq^{2r},q^{2r}/a;q^{2d})_k}
{(1+q^r)(aq^{2d},q^{2d}/a;q^{2d})_k}q^{3n(n-r)/d+dk^2+k(d-r)-n}\\
 &\equiv-(-1)^k\frac{(1+q^{2dk+r})(aq^{2r},q^{2r}/a;q^{2d})_k}
{(1+q^r)(aq^{2d},q^{2d}/a;q^{2d})_k}q^{dk^2+k(d-r)}\pmod{\Phi_n(-q)},
 \end{split}
 \end{equation*}
 which means  the left-hand side of (\ref{EQ1}) is congruent to $0$ modulo $\Phi_n(-q)$.
 Since the right-hand side of \eqref{EQ1} is also congruent to $0$ modulo $\Phi_n(-q)$, we
 can prove that \eqref{EQ1} holds modulo $\Phi_n(-q)(1-aq^{2n})(a-q^{2n})$.

The remaining proof of Theorem \ref{Thm2}
follows from the lines of the proof of Theorem \ref{Thm1}.
\end{proof}


\begin{thebibliography}{99}
\bibitem{Gasper}G.~Gasper, M.~Rahman,
 Basic hypergeometric series, second edition,
Encyclopedia of Mathematics and Its Applications \textbf{96},
Cambridge University Press, Cambridge, 2004.

\bibitem{Guo1}V.J.W. Guo, $q$-Analogues of the (E.2) and (F.2) supercongruences of van Hamme,
 Ramanujan J. \textbf{49} (2019), 531--544.

\bibitem{Guo}V.J.W. Guo, Some $q$-supercongruences from the Bailey transformation,
 Period. Math. Hungar. (2022).  https://doi.org/10.1007/s10998--021--00437--3.

\bibitem{Guo2}V.J.W. Guo and M.J. Schlosser, A new family of $q$-supercongruences modulo the fourth
power of a cyclotomic polynomial,
 Results Math. \textbf{75} (2020), Art. 155.

\bibitem{GuoZu}V.J.W. Guo and W. Zudilin, A $q$-microscope for supercongruences,  Adv. Math. \textbf{346} (2019), 329--358.

\bibitem{He}B. He, Some congruences on truncated hypergeometric series,  Proc. Amer. Math. Soc. \textbf{143} (2015), 5173--5180.

\bibitem{Liu1}Y. Liu and X. Wang, $q$-Analogues of two Ramanujan-type supergcongruences,  J. Math. Anal. Appl. \textbf{502} (2021), Art.~125238.

\bibitem{Liu_Wang1}
Y. Liu and X. Wang, Some $q$-supercongruences from a quadratic transformation by Rahman,
 Results Math. \textbf{77} (2022), Art. 44.


\bibitem{Long}L. Long, Hypergeometric evaluation identities and supercongruences,  Pacific J. Math. \textbf{249} (2011), 405--418.

\bibitem{Ni}H.-X. Ni and H. Pan, Divisibility of some binomial sums,  Acta Arith. \textbf{194} (2020), 367--381.

 \bibitem{Swisher}H. Swisher, On the supercongruence conjectures of van Hamme,  Res. Math. sci. \textbf{2} (2015), Art. 18.

 \bibitem{Tauraso}R. Tauraso, Some $q$-analogs of congruences for central binomial sums,  Colloq. Math. \textbf{133} (2013), 133--143.

\bibitem{Van}L. Van Hamme, Some conjectures concerning partial sums of generalized hypergeometric series,
in: $p$-Adic Functional Analysis (Nijmegen, 1996),
Lecture Notes in Pure and Appl. Math. \textbf{192}, Dekker, New York (1997), 223--236.

\bibitem{Wang_Yu}X. Wang and M. Yu,  Some generalizations of a congruence by Sun and Tauraso.
  Period. Math. Hungar. (2021).
 https://doi.org/10.1007/s10998--021--00432--8.


\bibitem{WeiChuanan}C. Wei, A further $q$-analogue of van Hamme's (H.2) supercongruence for any prime $p\equiv1\pmod4$,
Results Math. \textbf{76} (2021), Art. 92.

\bibitem{Wei2}C. Wei, Y. Liu and X. Wang, $q$-Supercongruences from the $q$-Saalsch{\"u}tz identity, Proc. Amer. Math. Soc. \textbf{149} (11) (2021), 4853--4861.


\end{thebibliography}
\end{document}